\documentclass[a4paper, 11pt]{amsart}

\usepackage[T1]{fontenc}
\usepackage[linktocpage]{hyperref}
	\hypersetup{colorlinks=true,linkcolor=blue,citecolor=blue,filecolor=magenta,urlcolor=blue}      
\usepackage{geometry}       
\usepackage[english]{babel}  
\usepackage{amsmath,amsfonts,amssymb,amsthm}

\newtheorem{theorem}{Theorem}[section]
\newtheorem{proposition}[theorem]{Proposition}
\newtheorem{lemma}[theorem]{Lemma}

\theoremstyle{definition}
\newtheorem{definition}[theorem]{Definition}
\newtheorem{example}[theorem]{Example}
\newtheorem{algorithm}[theorem]{Algorithm}
\newtheorem{problem}[theorem]{Problem}

\theoremstyle{remark}
\newtheorem{remark}[theorem]{Remark}

\newcommand{\CC}{\mathbb{C}}
\newcommand{\ZZ}{\mathbb{Z}}
\newcommand{\QQ}{\mathbb{Q}}
\newcommand{\FF}{\mathbb{F}}
\newcommand{\PP}{\mathbb{P}}
\newcommand{\A}{\mathcal{A}}
\renewcommand{\P}{\mathcal{P}}
\newcommand{\C}{\mathcal{C}}
\newcommand{\M}{\mathcal{M}}
\renewcommand{\L}{\mathcal{L}}
\newcommand{\ML}{\mathcal{ML}}
\newcommand{\FS}{\mathcal{FS}}

\DeclareMathOperator{\sing}{Sing}
\DeclareMathOperator{\Gal}{Gal}
\DeclareMathOperator{\PGL}{PGL}

\begin{document}

\title[Nonconnectedness of the moduli space, II: nonarithmetic pairs]{On the nonconnectedness of moduli spaces of arrangements,~II: construction of nonarithmetic pairs}

\author[B. Guerville-Ballé]{Beno{\^i}t Guerville-Ball{\'e}}
	\address{}
	\email{benoit.guerville-balle@math.cnrs.fr}

\thanks{}

\subjclass[2020]{
	51M15, 
	14N10, 
	14N20, 
	14H10, 
	51A45. 
}

\begin{abstract}
	Constructing lattice isomorphic line arrangements that are not lattice isotopic is a complex yet fundamental task. In this paper, we focus on such pairs but which are not Galois conjugated, referred to as nonarithmetic pairs. Splitting polygons have been introduced by the author to facilitate the construction of lattice isomorphic arrangements that are not lattice isotopic. Exploiting this structure, we develop two algorithms which produce nonarithmetic pairs: the first generates pairs over a number field, while the second yields pairs over the rationals. Moreover, explicit applications of these algorithms are presented, including one complex, one real, and one rational nonarithmetic pair.
\end{abstract}

\maketitle

\section{Introduction}

The concept of moduli space plays an important role in various branches of mathematics, serving as a fundamental framework for understanding families of geometric objects. Moduli spaces provide a structured way to classify objects up to isomorphism, capturing the intrinsic parameters that define their equivalence classes. Understanding their structure and behavior is crucial, not only from the theoretical persepective but also for applications in algebraic geometry, topology, or mathematical physics. However, the study of moduli spaces is a challenging task due to their complex nature. Indeed, Mn\"ev Universality Theorem~\cite{Mnev:Universality} and its numerous extensions by Vakil~\cite{Vak:Murphy_law}, imply that, in lots of cases, moduli spaces can behave as badly as one can imagine.

\medskip

The \emph{moduli space} $\M(\A)$ of a line arrangement $\A$ in $\CC\PP^2$ can be defined as the set of all arrangements that are lattice isomorphic to $\A$, quotiented by the natural action of $\PGL_3(\CC)$. MacLane~\cite{MacLane} showed that such moduli spaces are not necessarily path-connected, while Mn\"ev~\cite{Mnev:Universality} proved that any singularity defined over $\ZZ$ appears in at least one moduli space, revealing the complex nature mentioned above. Randell's Isotopy Theorem~\cite{Ran:isotopy} states that two path-connected arrangements are topologically equivalent, illustrating the importance of understanding the topology of such moduli spaces. The moduli spaces of line arrangements with few lines have been classified from different perspectives. The works of Nazir and Yoshinaga~\cite{NazYos} and of Fei~\cite{Fei} classify the topological types of arrangements up to 9 lines. Recently, Corey and Luber classified the diffeomorphic types of line arrangements with up to $12$ lines, see~\cite{Cor:Grassmannians, CorLub:Grassmannian, CL:singular_moduli}. The interplay between connectedness and combinatorics has been studied by the author and Viu-Sos in~\cite{Gue:splitting_polygons, GBVS:connectedness}. The present paper is a continuation of these studies.

\subsection{Arithmetic, nonarithemtic and rational pairs}

Our study focuses on the connected components of the moduli spaces, so we can assume that any arrangement is defined over a number field. Indeed, due to the algebraic nature of the moduli space of a fixed arrangement $\A$, any arrangement in $\M(\A)$ is lattice isotopic to an arrangement defined over a number field. The number field generated by all the coefficients of the lines of $\A$ is called the \emph{definition field} of $\A$ and is denoted by $\FF(\A)$. Throughout this paper, a \emph{pair} of arrangements always refers to lattice isomorphic arrangements which are not isotopy equivalent. In other words, they are in different connected components of their shared moduli space.

\begin{definition}\label{def:pairs}
	Let $\A_1$ and $\A_2$ be a pair of arrangements, and let $\FF$ be a number field that contains both $\FF(\A_1)$ and $\FF(\A_2)$.
	\begin{itemize}
		\item The pair is \emph{arithmetic} if there exists a Galois automorphism $\sigma\in\Gal(\FF/\QQ)$ such that $\sigma \cdot \A_1 = \A_2$. It is \emph{nonarithmetic} otherwise.
		\item The pair is \emph{rational} if $\FF(\A_1)=\FF(\A_2)=\QQ$.
	\end{itemize}
\end{definition}

\begin{remark}
	Obviously, a rational pair is a nonarithmetic pair.
\end{remark}

\begin{example}\label{ex:pairs}\mbox{}
	\begin{enumerate}
		\item The MacLane arrangements~\cite{MacLane} or the Nazir-Yoshinaga arrangements~\cite[Example~5.3]{NazYos} form arithmetic pairs where the Galois automorphism $\sigma$ is the complex conjugation.
		\item The Rybnikov arrangements~\cite[Section~4]{Ryb} form a complex nonarithmetic pair.
		\item The arrangements $\A^{1,1}$ and $\A^{1,-1}$ in~\cite[Section~2.2]{GBVS:configurations} form a rational pair.
	\end{enumerate}
\end{example}

\subsection{Arithmetic pairs}

Due to the algebraic nature of the moduli space, arithmetic pairs are simple to construct, e.g.~\cite{MacLane, NazYos, ACCM, Gue:LLN}. Furthermore, they are a particular case of conjugate varieties. It is known that such varieties can have different topologies, see~\cite{Abelson,Serre} for the general case and~\cite{ACCM} for line arrangements --see also~\cite{Gue:ZP,Cad:PhD} for numerous other examples. Nevertheless, both the topology and the geometry of such pairs are difficult to distinguish. Indeed, any algebraic topological invariant will fail to distinguish a topological difference. In particular, the profinite completion of the fundamental group of their complements are isomorphic. Similarly, the geometric line operator $\Lambda_{\mathfrak{m},\mathfrak{n}}$ introduced by Rouleau in~\cite{Rou2,Rou1} cannot distinguish a difference of geometry in such pairs.

\subsection{Nonarithmetic pairs}

There is a wide class of problems in the study of arrangements and, more generally, of plane curves, asking whether a given property --be it geometric, topological, or otherwise-- is determined by the local type of the singularities or not. To answer negatively such a question, the most efficient way is to provide an explicit counterexample. When considering two arrangements, the more different they are, the more likely they are to form a counterexample. From this observation and according to the previous paragraph, it is natural, in a first place, to consider arrangements that form a nonarithmetic pair. It is not a coincidence if the first examples of Zariski pair\footnote{A Zariski pair is form by two lattice isomorphic arrangements that have nonhomeomorphic embedded topology.} for curves~\cite{Zar1,Zar2} or for arrangements~\cite{Ryb} are both nonarithmetics. However, one of the challenges remains in constructing such nonarithmetic pairs.

In~\cite{Gue:splitting_polygons}, the author introduced the combinatorial notion of the \emph{splitting polygon pattern}. It is a sub-structure of a line combinatorics $\C$ suggesting that the moduli space $\M(\C)$ is not connected. This structure is present in all the examples given in Example~\ref{ex:pairs}. It has been used in~\cite{Gue:splitting_polygons}, to produce several examples of arithmetic Zariski pairs. The purpose of the current paper is to present and illustrate two algorithms for constructing nonarithmetic pairs using the splitting polygon structure. The first algorithm produces nonarithmetic pairs over a number field, while the second generates rational pairs.

\subsection{Structure of the paper}

In Section~\ref{sec:splitting_and_algo}, we recall the definition of the splitting polygon structure. Then, we describe the two algorithms (Algorithm~\ref{algo:non_arithmetic} and~\ref{algo:rational}) which produce nonarithemtic pairs over a number field and rational pairs, respectively. Section~\ref{sec:applications} is dedicated to applications of these algorithms with the production of explicit examples (Theorems~\ref{thm:complex_nonarithmetic_pair}, \ref{thm:real_nonarithmetic_pair} and~\ref{thm:rational_nonarithmetic_pair}). To conclude, in Section~\ref{sec:discussion}, we discuss the limitations of the algorithms and some possible improvments, question the existence of nonarithmetic pairs with less than $13$ lines (Problem~\ref{pb:nonarithmetic}), and wonder if there is a topological difference in the examples procuded in Section~\ref{sec:applications} (Problem~\ref{pb:ZP}).

\section{Splitting polygon and algorithms}\label{sec:splitting_and_algo}

In order to state the algorithms, one needs to recall the definition and some properties of the splitting polygon structure introduced in~\cite{Gue:splitting_polygons}.

\medskip

As originally defined in~\cite{ACCM}, a \emph{line combinatorics} is the data of a finite set $\L$ and a subset $\P$ of the powerset of $\L$ such that:
\begin{itemize}
	\item for all $P\in\P$, one has $|P|\geq 2$,
	\item for all $i,j\in\P$, with $i\neq j$, there exists a unique $P\in\P$ such that $\{i,j\}\subset P$.
\end{itemize}

\medskip

Let $\A$ be a line arrangement of $\CC\PP^2$, and let $\sing(\A)$ be the set of singular points of $\A$, where each singular point $P$ is given as the maximal subset of $\A$ formed by the lines of $\A$ passing through $P$. Under such a convention, $\big(\A,\sing(\A)\big)$ is a line combinatorics, named the combinatorics of $\A$ and denoted by $\C(\A)$ --or simply $\C$ if no confusion is possible.

\subsection{The splitting polygon structure}

\begin{definition}\label{def:plinth}
	A \emph{plinth} $\Psi$ of length $r\geq3$ on a line combinatorics $\C$ is the data of two tuples $(S_1,\dots,S_r)\subset\L$ and $(P_1,\dots,P_r)\subset\P$, named respectively the \emph{support} and the \emph{pivot points}, such that for all pivot points $P_i$, we have $S_i\notin P_i$ and $S_{i+1}\notin P_i$, where the indices are considered modulo $r$.
\end{definition}

Let $\A$ be an arrangement with combinatorics $\C$. The connected component of $\M(\C)$ containing the class of $\A$ is denoted by $\M(C)^{\A}$. If $\Psi$ is a plinth on $\C$ then the geometrical image of $\Psi$ is called the realization of $\Psi$ in $\A$, it will be denoted by $\Psi_{\A}$.

\begin{definition}\label{def:rigid_projective_system}
	A plinth $\Psi$ on a combinatorics $\C$ forms a \emph{rigid projective system} of the connected component $\M(\C)^{\A}$ if, for any $\A'\in\M(\C)^{\A}$, there exsits a projective transformation $\tau$ which sends $\Psi_{\A}$ on~$\Psi_{\A'}$.
\end{definition}

The following proposition is a direct consequence of the definiton.

\begin{proposition}\label{propo:rigid}
	If $\dim_\CC(\M(\C)^\A)=0$, then any plinth $\Psi$ on $\C$ is a rigid projective system of~$\M(\C)^\A$.
\end{proposition}

\begin{definition}\label{def:splitting_polygons}
	Let $\A$ be a line arrangement and $\Psi$ be a plinth on $\C$ with support $(S_1,\dots,S_r)$ and pivot points $(P_1,\dots,P_r)$. A set of $r$ lines $E^\lambda=(E_1^\lambda,\dots,E_r^\lambda)\subset (\check{\CC\PP}^2)^r$ is a \emph{splitting polygon} on the plinth $\Psi$ if, for all $i\in\{1,\dots,r\}$, it verifies the following conditions --where all the indices are considered modulo $r$.\footnote{In the particular case $r=3$, we use the term \emph{splitting triangle}.}
	\begin{enumerate}
		\item $E_i^\lambda \cap E_{i+1}^\lambda \cap S_i \neq \emptyset$,
		\item $E_i$ passes through the pivot point $P_i$,
		\item all the other intersections are generic.
	\end{enumerate}
	If $E^\lambda$ is a splitting polygon on $\Psi$, we denote by $\A_\Psi^\lambda$ the union of $\A$ and $E^\lambda$, and by $\C_\Psi$ it combinatorics.
\end{definition}

A set of lines $E^\lambda$ is a \emph{nonsplitting polygon} if it verifies all the conditions of the splitting polygon but fails to verify a unique condition~(1). The main result of~\cite{Gue:splitting_polygons} is the following.

\begin{theorem}[\cite{Gue:splitting_polygons}]\label{thm:splitting_polygons}
	Let $\A$ be a line arrangment with combinatorics $\C$, and let $\Psi$ be a rigid projective system of $\M(\C)^\A$. If it exists two distinct splitting polygons $E^{\lambda_1}$ and $E^{\lambda_2}$, and a nonsplitting polygon then $\A_\Psi^{\lambda_1}$ and $\A_\Psi^{\lambda_2}$ are in distinct connected components of $\M(\C_\Psi)$.
\end{theorem}

\subsection{The polynomial $\Delta_\Psi$}\label{sec:polynomial_Delta}

The existence of such splitting polygons and nonsplitting polygon is related with the degree and the roots of a polynomial $\Delta_\Psi$ ensuring the realization or not of condition~(1) in Definition~\ref{def:splitting_polygons}. This polynomial can be constructed as follows.

Let $Q_1^\lambda$ be a generic point of the line $S_1$ which is different from $S_1 \cap S_2$. It can be linearly parametrized by a unique parameter $\lambda \in \CC$. By Definition~\ref{def:plinth}, $S_i\notin P_i$, so one can define $E_1^\lambda$ as the line passing through $Q_1^\lambda$ and $P_1$. We denote by $Q_2^\lambda$ the intersection point of $E_1^\lambda$ and $S_2$. This point is well defined due to the condition $S_{i+1}\notin P_i$ in Definition~\ref{def:plinth}. Similarly, we construct iteratively the lines $E_i^\lambda$ and the points $Q_{i+1}^\lambda$, for $i\in\{2,\dots,r\}$. We denote by $\Delta_\Psi$ the determinant of the $3\times3$ matrix formed by the coefficients of lines $S_1$, $E_1^\lambda$ and $E_r^{\lambda}$. If $\FF(\A)$ is a definition field of $\A$, then $\Delta_\Psi$ is an element of the polynomial ring $\FF(\A)[\lambda]$. By construction, $\Delta_\Psi$ is of degree at most~$2$.

\medskip

Assume that $\Delta_\Psi$ is of degree~$2$. If it is also irreducible in $\FF(\A)[\lambda]$, then all the conditions of Theorem~\ref{thm:splitting_polygons} are verified, see~\cite[Theorem~2.8]{Gue:splitting_polygons}. Such splitting polygons will be called \emph{irreducible splitting polygons} and \emph{reduced splitting polygons} otherwise. In such situation, the created pair will be an arithmetic one. To create a nonarithmetic pair, we will search for reduced splitting polygons. It is worth to notice that the condition ``$\Delta_\Psi$ is reducible in $\FF(\A)$'' is necessary but not sufficient to create nonarithemtic pair. Indeed, in that case, one also needs to verify Condition~3 of Definition~\ref{def:splitting_polygons}

\begin{lemma}\label{lem:reduced_polynomial}
	If the polynomial $\Delta_{\Psi}$ is reducible then there exists a nonsplitting polygon on $\Psi$.
\end{lemma}

\begin{proof}
	Since $\Delta_\Psi$ is of degree at most~$2$, the hypothesis that it is reducible induces that it is of degree exactly~$2$. The conditions for the tuple $E^{\lambda_0}$ to form a nonsplitting polygon are: $\Delta_\Psi(\lambda_0) \neq 0$ and all the other intersections are generics. The first condition is Zariski open, while the second kind are linear Zariski-closed conditions. Indeed these conditions correspond to a line, linearly parametrized by $\lambda$, that passes through a point independent of $\lambda$. Since the intersection of an Zariski-open subset with finitely many proper Zariski-closed subsets is not empty then such $\lambda_0$ exists.
\end{proof}

\subsection{The algorithms}

The first example that comes in mind when we talk about nonarithmetic pairs is the Rybnikov arrangements. We have shown in~\cite[Section.~4.1]{Gue:splitting_polygons} that they can be constructed by adding successively two splitting triangles such that the first one is irreducible and the second one is reducible. Taking inspiration from this construction, nonarithmetic pairs can be constructed using the following algorithm.

\begin{algorithm}\label{algo:non_arithmetic}
	Construction of nonarithmetic pair over a number field.
	\begin{enumerate}
		\item Consider a rigid arrangement $\A$ with definition field $\FF(\A)$ (eventually $\FF(\A)=\QQ$).
		\item Enumerate all the plinths of $\C(\A)$ with a fixed length.
		\item For each plinth $\Psi_1$ of $\C(\A)$: compute the polynomial $\Delta_{\Psi_1}$ and determine its irreducibility in $\FF(\A)[\lambda]$.
		\item If the polynomial $\Delta_{\Psi_1}$ is irreducible over $\FF(\A)$, then construct the arithmetic pair $\A_{\Psi_1}^1$ and $\A_{\Psi_1}^2$ by adding to $\A$ a splitting polygon on the plinth $\Psi_1$.
		\item Enumerate all the plinths of $\C(\A^1_{\Psi_1})=\C(\A^2_{\Psi_1})$ with a fixed lenght.
		\item Fix $i\in\{1,2\}$ and for each plinth $\Psi_2$ on $\C(\A^i_{\Psi_1})$: compute the polynomial $\Delta_{\Psi_2}$ and determine its irreducibility in $\FF(\A^i_{\Psi_1})[\lambda]$.
		\item If the polynomial $\Delta_{\Psi_2}$ is reducible in $\FF(\A_{\Psi_1}^i)[\lambda]$, verify if, for each roots $\lambda_j$ of $\Delta_{\Psi_2}$, the tuple of lines $E^{\lambda_j} = (E_1^{\lambda_j},\dots,E_r^{\lambda_j})$ form a splitting polygon on the plinth~$\Psi_2$.
		\item The arrangements $\A_{\Psi_1,\Psi_2}^{i,1}$ and  $\A_{\Psi_1,\Psi_2}^{i,2}$ obtained by adding to $\A_{\Psi_1}^i$ a splitting polygon on the plinth $\Psi_2$ form a nonarithmetic pair.
	\end{enumerate}
\end{algorithm}

\begin{proof}
	By construction, the definition field $\FF(\A_{\Psi_1}^1)=\FF(\A_{\Psi_1}^2)$ is a quadratic extension of $\FF(\A)$. We denote by $\sigma$ the generator of the associated Galois group. The automorphism $\sigma$ fixes line-by-line the arrangement $\A$, so one has $\sigma \cdot \A_{\Psi_1}^1 = \A_{\Psi_1}^2$.\footnote{Let $\sigma$ be a Galois automorphism of a number field $\FF$. If a line $L:ax+by+cz=0$ is defined over $\FF$, then we denote $\sigma\cdot L$ the line defined by the equation $\sigma(a)x + \sigma(b)y + \sigma(c)z = 0$. The arrangement $\sigma\cdot\A$ is defined as $\{ \sigma\cdot L \mid L\in\A \}$.} Since $\A$ is rigid then Proposition~\ref{propo:rigid} implies that $\Psi_1$ is a rigid projective system of $\M(\A)^\A$, so by Theorem~\ref{thm:splitting_polygons} the arrangements $\A_{\Psi_1}^1$ and $\A_{\Psi_1}^2$ form an arithmetic pair. Furthermore, by construction these arrangements are also rigid.

	Assume that $i=1$ (the case $i=2$ is similar). As noted above, the arrangement $\A_{\Psi_1}^1$ is rigid so, by Proposition~\ref{propo:rigid}, $\Psi_2$ is a rigid projective system of $\M(\A_{\Psi_1}^1)^{\A_{\Psi_1}^1}$. Additionally, since the polynomial $\Delta_{\Psi_2}$ is reducible, then Lemma~\ref{lem:reduced_polynomial} implies that a nonsplitting polynomial exists, and that $\Delta_{\Psi_2}$ is of degree $2$. We denote by $\lambda_1$ and $\lambda_2$ its two roots. Since $E^{\lambda_1}$ and $E^{\lambda_2}$ form splitting polygons on the plinth $\Psi_2$, then, by Theorem~\ref{thm:splitting_polygons}, the arrangements $\A_{\Psi_1,\Psi_2}^{1,1}$ and  $\A_{\Psi_1,\Psi_2}^{1,2}$ form a pair.

	Since $\Delta_{\Psi_2}$ is reducible in $\FF(\A_{\Psi_1}^1)[\lambda]$ then $\FF(\A_{\Psi_1,\Psi_2}^{1,i})=\FF(\A_{\Psi_1}^1)$. Additionally, since $\sigma \cdot \A_{\Psi_1}^1 = \A_{\Psi_1}^2$, then $\sigma \cdot \A_{\Psi_1,\Psi_2}^{1,1} \neq \A_{\Psi_1,\Psi_2}^{1,2}$. So $\A_{\Psi_1,\Psi_2}^{1,1}$ and  $\A_{\Psi_1,\Psi_2}^{1,2}$ do not form an arithmetic pair.
\end{proof}

\begin{remark}
	It is also possible to consider a given arithmetic pair and to apply Algorithm~\ref{algo:non_arithmetic} from Step~(5).
\end{remark}

\begin{remark}\label{rmk:no_reduced_plinth}
	At Step~(3) (resp. Step~(6)), if there is no plinth with an irreducible (resp. reducible) splitting polynomial then one can add a line to $\A$ (resp. $\A_{\Psi_1}^i$'s) passing through at least $2$ singular points. Then we restart the algorithm at Step~(2) (resp. Step~(5)). The additional line will neither modify the definition field nor the rigidity of the considered arrangements.
\end{remark}

One can modify the previous algorithm to produce rational pairs. The proof is similar than for Algorithm~\ref{algo:non_arithmetic}.

\begin{algorithm}\label{algo:rational}
	Construction of rational pair.
	\begin{enumerate}
		\item Consider a rigid arrangement $\A$ defined over $\QQ$.
		\item Enumerate all the plinths of $\C(\A)$.
		\item For each plinth $\Psi$ on $\C(\A)$: compute the polynomial $\Delta_{\Psi}$ and determine its irreducibility.
		\item If the polynomial $\Delta_{\Psi}$ is reducible in $\QQ[\lambda]$, then verify if, for its two roots $\lambda_i$ of $\Delta_{\Psi}$, the tuple of lines $E^{\lambda_i}=(E_1^{\lambda_i},\dots,E_r^{\lambda_i})$ form splitting polygons on the plinth $\Psi$.
		\item The arrangements $\A_\Psi^1$ and $\A_\Psi^2$ form a rational pair.
	\end{enumerate}
\end{algorithm}

\section{Applications of the algorithms}\label{sec:applications}

The first two examples demonstrate applications of Algorithm~\ref{algo:non_arithmetic}, generating nonarithmetic pairs defined over a number field --a complex and a real one. The first example utilizes the MacLane arrangements, and the second uses the Falk-Sturmfels arrangements. The third example is a rational pair produced by Algorithm~\ref{algo:rational}.

\subsection{A complex nonarithmetic pair}

The first example is inspired from Rybnikov's construction but restraining to rigid arrangements. Steps (1)--(4) of Algorithm~\ref{algo:non_arithmetic} correspond to the construction of the MacLane arrangements using the splitting polygon structure. Let us recall the data needed in this construction, for details we refer to~\cite[Section~3.1]{Gue:splitting_polygons}.

Let $\A$ be the rigid arrangement with definition field $\FF(\A)=\QQ$ and defined by
\begin{equation*}
	\begin{array}{l p{1cm} l p{1cm} l}
		L_1: z = 0, && L_2: x = 0, && L_3: x - z = 0, \\
		L_4: y = 0, && L_5: y - z = 0. &&
	\end{array}
\end{equation*}
The plinth $\Psi_1$ is given by the support $(L_1, L_2, L_4)$ and the pivot points $(L_3 \cap L_4, L_3 \cap L_5, L_2 \cap L_5)$. The MacLane arrangements are defined over $\QQ(\zeta_3)$, with $\zeta_3$ a primitive third root of unity. The splitting triangles are given by:
\begin{equation*}
	E_1^{\lambda}: x + \lambda y - z = 0,\quad E_2^{\lambda}: (\lambda - 1)x - \lambda y + z = 0,\quad E_3^{\lambda}: (\lambda - 1)x - y + z = 0,
\end{equation*}
for $\lambda\in\{\zeta_3, \zeta_3+1\}$. The arrangements $\A_{\Psi_1}^i$ will be denoted by $\ML_i$ for simplicity.

From now on, we will only deal with $\ML_1$; a similar work can be done with $\ML_2$. The lines $E_1^{\zeta_3},E_2^{\zeta_3},E_3^{\zeta_3}$ are denoted $L_6,L_7,L_8$ respectively. In Step (6), no plinth $\Psi_2$ produces a reduced polynomial $\Delta_{\Psi_2}$ in $\QQ(\zeta_3)[\lambda]$. So, according to Remark~\ref{rmk:no_reduced_plinth}, we add a line passing through $2$ singular points. Unfortunately, even with such an additional line, there is still no plinth that produces a reduced polynomial $\Delta_{\Psi_2}$. So, we add a second line passing through $2$ singular points. Thus, we consider the lines $L_{9}$ passing through $P_1 = L_2 \cap L_6 \cap L_7$ and $Q_1 = L_3 \cap L_8 $, and $L_{10}$ passing through $P_2 = L_4 \cap L_9$ and $Q_2 = L_3 \cap L_5 \cap L_7$. Their equations are given by:
\begin{equation*}
	L_9: 3x + (2\zeta_3 - 1)y - (\zeta_3 + 1)z = 0, \quad L_{10}: 3x + (\zeta_3 - 2)y - (\zeta_3 + 1)z = 0.
\end{equation*}

Let $\Psi_2$ be the plinth with support $(L_1, L_3, L_4)$ and pivot points $(L_8 \cap L_{10}, L_1 \cap L_{10}, L_5 \cap L_9)$. The polynomial $\Delta_{\Psi_2}$ is $(2\zeta_3 + 2) (\lambda + \frac{\zeta_3-2}{3}) (\lambda + \frac{-\zeta_3+1}{2})$. The splitting polygons associated are:
\begin{align*}
	& 			L^1_{11}: (-2\zeta_3 + 4)x + (2\zeta_3 - 1)y - (\zeta_3/2 + 2)z = 0,
	\quad 	L^1_{12}: (-\zeta_3 + 2)x + (\zeta_3 - 1)y - 1/2z = 0,\\
	& 			L^1_{13}: 6x + (3\zeta_3 - 3)y - (\zeta_3 + 1)z = 0,
\end{align*}
and
\begin{align*}
	&			L^2_{11}: (-\zeta_3 + 2)x + (-\zeta_3 + 1)y - 2z = 0,
	\quad	L^2_{12}: 3x + (\zeta_3 - 2)y + (2\zeta_3 - 4)z = 0, \\
	&			L^2_{13}: (\zeta_3 + 1)x + y - 2z = 0.
\end{align*}

By a direct application of Algorithm~\ref{algo:non_arithmetic}, one has the following.

\begin{theorem}\label{thm:complex_nonarithmetic_pair}
	The arrangements $\ML_1\cup\{L_9,L_{10}\}\cup\{L^1_{11},L^1_{12},L^1_{13}\}$ and $\ML_1\cup\{L_9,L_{10}\}\cup\{L^2_{11},L^2_{12},L^2_{13}\}$ form a complex nonarithmetic pair defined over $\QQ(\zeta_3)$.
\end{theorem}

\begin{remark}
	We applied Algorithm~\ref{algo:non_arithmetic} with brut force approach --i.e. without considering the symmetries of the line combinatorics of the MacLane arrangement. It generates 83,320 plinths of length $3$. Note that we removed the plinths having concurrent lines in their supports since they always produce polynomial $\Delta_\Psi$ of degree at most~$1$. From these plinths, we produce $42$~nonarithmetic pairs defined over $\QQ(\zeta_3)$, up to lattice isomorphism.
\end{remark}

\subsection{A real nonarithmetic pair}

To construct a real nonarithmetic pair, we use the Falk-Sturmfels arrangements $\FS_1$ and $\FS_2$ defined over $\QQ(\sqrt{5})$ by the following equations, where $\phi=\frac{-1\pm\sqrt{5}}{2}$:
\begin{equation*}
	\begin{array}{lll}
		L_1: z = 0, \quad & L_2: x = 0, \quad &  L_3: x - z = 0, \\
		L_4: y = 0, \quad &  L_5: y - z = 0, \quad &  L_6: x - y = 0, \\
		L_7: x + \phi y - z = 0, \quad &  L_8: \phi x - \phi y + z = 0, \quad &  L_9: -\phi x + (\phi - 1)y = 0.
	\end{array}
\end{equation*}

By~\cite[Section~3.2.1]{Gue:splitting_polygons}, they can be constructed using the splitting polygon structure, with an irreducible polynomial $\Delta_{\Psi_1}$. So Steps~(1)-(4) in Algorithm~\ref{algo:non_arithmetic} are already done. Similarly to the previous example, there is no reduced plinth on $\C(\FS_1)$, so according to Remark~\ref{rmk:no_reduced_plinth}, we can add a line $L_{10}$ through the singular points $L_1 \cap L_2 \cap L_3$ and $L_6 \cap L_7$. It is defined by the equation $L_{10}: x - \phi z = 0$.

We consider the plinth $\Psi_2$ given by the support $(L_1, L_2, L_5)$ and pivot points $(L_5 \cap L_7, L_4 \cap L_{10}, L_4 \cap L_8)$. The polynomial $\Delta_{\Psi_2}$ is $(-\phi - 2) (\lambda + \phi)  (\lambda - \phi + 2)$, and the associated splitting polygons are:
\begin{align*}
	& L^1_{11}: (\phi + 2)x - (\phi + 1)y + \phi z = 0, \quad  L^1_{12}: (\phi - 1)x - \phi y + (2\phi - 1)z = 0, \\
	& L^1_{13}: (-2\phi + 1)x + (-3\phi + 2)y + (\phi - 1)z = 0,
\end{align*}
and
\begin{align*}
	& L^2_{11}: (\phi + 2)x - (\phi + 3)y + (\phi + 2)z = 0, \quad  L^2_{12}: -(\phi + 1)x + (\phi - 2)y + z = 0, \\
	& L^2_{13}: -x + (-\phi + 2)y - (\phi + 1)z = 0.
\end{align*}

Using Algorithm~\ref{algo:non_arithmetic}, one can deduce the next theorem.

\begin{theorem}\label{thm:real_nonarithmetic_pair}
	The arrangements $\FS_1\cup\{L_{10}\}\cup\{L^1_{11},L^1_{12},L^1_{13}\}$ and $\FS_1\cup\{L_{10}\}\cup\{L^2_{11},L^2_{12},L^2_{13}\}$ form a real nonarithmetic pair defined over the $\QQ(\sqrt{5})$.
\end{theorem}

\begin{remark}
	In the application of Algorithm~\ref{algo:non_arithmetic}, the brut force approach generates $83,166$ plinths of length $3$ and produces $91$ nonarithmetic pairs, up to lattice isomorphism.
\end{remark}

\subsection{A rational pair}

To complete our set of examples, we now apply Algorithm~\ref{algo:rational} to construct a rational pair.  Consider the rational arrangement $\A$ defined by the lines
\begin{equation*}
	\begin{array}{lll}
		L_1: 2y - z = 0, \quad &   L_2: x - y = 0, \quad &  L_3: x + y - z = 0, \\
		L_4: x = 0, \quad &  L_5: 2x - 2y + z = 0, \quad &  L_6: x - z = 0, \\
		L_7: 2x + 6y - 5z = 0,   \quad & L_8: y = 0, \quad &  L_9: y - z = 0, \\
		L_{10}: z = 0.
	\end{array}
\end{equation*}
Let $\Psi$ be the plinth with support $(L_1, L_2, L_7)$ and pivot points $(L_3 \cap L_6 \cap L_8, L_1 \cap L_8 \cap L_9 \cap L_{10}, L_5 \cap L_6)$. The polynomial $\Delta_\Psi$ is $(-6)(\lambda - 5/2) (\lambda - 2/3)$. The two associated splitting polygons are:
\begin{equation*}
	L^1_{11}: 3x + 2y - 3z = 0, \quad  L^1_{12}: 5y - 3z = 0, \quad  L^1_{13}: 6x - 2y - 3z = 0.
\end{equation*}
and
\begin{equation*}
	L^2_{11}: x - 3y - z = 0, \quad  L^2_{12}: 2y + z = 0, \quad  L^2_{13}: 4x + 6y - 13z = 0.
\end{equation*}

By Algorithm~\ref{algo:rational}, one has the following.

\begin{theorem}\label{thm:rational_nonarithmetic_pair}
	The arrangements  $\A\cup\{L^1_{11},L^1_{12},L^1_{13}\}$ and $\A\cup\{L^2_{11},L^2_{12},L^2_{13}\}$ form a rational pair.
\end{theorem}

\begin{remark}
	In the application of Algorithm~\ref{algo:rational}, the brut force approach generates $102,517$ plinths of length $3$, and produces $116$ other rational pairs, up to lattice isomorphism.
\end{remark}

\section{Discussion}\label{sec:discussion}

\subsection{The algorithms}

The algorithms presented in this paper have proven to be effective, as demonstrated by the numerous examples of nonarithmetic pairs generated. However, there are limitations that warrant further investigation. One significant area for improvement is the generation of plinths. Currently, the process enumerates all possible plinths using a force brut approach, which can be computationally intensive for arrangements with a high number of lines. Taking advantage of the symmetries of the arrangement could streamline this process, reducing redundancy and improving efficiency. Additionally, the question arises whether there exists combinatorial conditions on plinths that consistently produce nonarithmetic pairs. Identifying such conditions could not only enhance our understanding of the underlying geometric structures but also lead to more targeted and efficient algorithms for generating nonarithmetic pairs.

\subsection{Nonarithmetic pairs with at most $12$ lines}

The Rybnikov arrangements and the rational Zariski pairs given in~\cite{GBVS:configurations} each consist of at least $13$ lines. Notably, all the examples produced in this paper also contain $13$ lines. In our process, it appeared to be necessary to add some lines due to the absence of a reduced plinth. Consequently, we were unable to construct nonarithmetic pairs with fewer than $13$ lines. This observation naturally leads to the following problem.
\begin{problem}\label{pb:nonarithmetic}
	Construct a pair with at most $12$ lines which is not lattice isotopic to a nonarithmetic pair or prove that such a pair does not exist.
\end{problem}

\subsection{Topology of the examples}

In~\cite{Gue:splitting_polygons}, the author constructed several arithmetic pairs by successively applying two splitting triangles, akin to Algorithm~\ref{algo:non_arithmetic}. Some of these pairs form Zariski pairs with nonisomorphic fundamental groups. In contrast, the topology of all examples constructed in the present paper cannot be distinguished using the truncated Alexander test~\cite{ACCM:Rybnikov,ACGM}. The author also examined several other topological invariants, such as the loop-linking numbers~\cite{Cad:PhD,Gue:LLN}, the torsion of the lower central series factors~\cite{ArtGueViu}, and the torsion in the first Chen groups. None of these tests could demonstrate that the examples have nonhomeomorphic embedded topology.

\begin{problem}\label{pb:ZP}
	Determine if the examples presented in this paper have nonhomeomorphic embedded topology.
\end{problem}

%
%

\bibliographystyle{plain}
\bibliography{biblio}

\end{document}